\documentclass[11pt]{article}
\usepackage{amsmath}
\usepackage{mathdots}
\usepackage{amssymb}
\usepackage{amscd}
\usepackage[matrix,arrow]{xy}
\usepackage{amsmath,amsthm,amsfonts}
\usepackage{color}
\usepackage{bbm}

\theoremstyle{plain}
\newtheorem{theo}{Theorem}[section]
\newtheorem{lemma}[theo]{Lemma}
\newtheorem{prop}[theo]{Proposition}

\newtheorem{cor}[theo]{Corollary}

\numberwithin{equation}{section}

\theoremstyle{remark}

\newtheorem{remark}[theo]{Remark}

\theoremstyle{definition}

\def\C{\mathbb{C}}
\def\Z{\mathbb{Z}}

\def\PP{{\mathbb P}}

\def\pre-tr{\operatorname{pre-tr}}

\def\Hom{\operatorname{Hom}}
\def\End{\operatorname{End}}
\def\gr{\operatorname{gr}}

\newcommand{\Le}{{\mathbb L}}

\newcommand{\bbQ}{{\mathbb Q}}

\newcommand{\KM}{{\mathcal K \mathcal M}}
\newcommand{\CH}{{\mathcal C \mathcal H}}

\newcommand{\cO}{{\mathcal O}}

\newcommand{\cK}{{\mathcal K}}
\newcommand{\cH}{{\mathcal H}}

\newcommand{\Ker}{\operatorname{Ker}}

\newcommand{\Ext}{\operatorname{Ext}}

\newcommand{\Spec}{\operatorname{Spec}}

\newcommand{\rk}{\operatorname{rk}}

\newcommand{\Map}{\operatorname{Map}}

\newcommand \uno {{\mathbbm 1}}

\title{Integral Chow motives of threefolds with $K$-motives of unit type\footnotetext{This work is supported by the RSF under a grant 14-50-00005.}}

\author{Sergey Gorchinskiy\\ \\
\small{Steklov Mathematical Institute of Russian Academy of Sciences, Moscow, Russia}\\
\small{e-mails: {\tt gorchins@mi.ras.ru}}}

\date{}

\begin{document}

\maketitle

\begin{abstract}
We prove that if a smooth projective algebraic variety of dimension less or equal to three has a unit type integral $K$-motive, then its integral Chow motive is of Lefschetz type. As a consequence, the integral Chow motive is of Lefschetz type for a smooth projective variety of dimension less or equal to three that admits a full exceptional collection.
\end{abstract}

\section{Introduction}

There are various categories of Grothendieck motives of smooth projective algebraic varieties. A~category of motives depends on the choice of a global intersection theory, see Manin's exposition in~\cite[\S\,1]{Ma}. Among these categories, we have the category of Chow motives and the category of $K$-motives. The Chow motive of a smooth projective variety $X$ is controlled by algebraic cycles, more precisely, by Chow groups of products of $X$ with other varieties. The $K$-motive of~$X$ is controlled by vector bundles, more precisely, by $K_0$-groups of products of $X$ with other varieties. It makes sense to compare these two motives of~$X$.

Simplest Chow motives are that of Lefschetz type, that is, direct sums of tensor powers of the Lefschetz motive. Simplest $K$-motives are that of unit type, that is, direct sums of the unit object, see Section~\ref{sec:main} for more detail.

It was shown by the author and Orlov in~\cite[Prop.\,4.2]{GO} that if the Chow motive of~$X$ is of Lefschetz type, then the $K$-motive of $X$ is of unit type. A natural question is then whether the converse implication is also true:

\medskip

\begin{itemize}
\item[]{}
{\it {\sc Question.} Let $X$ be a smooth projective variety such that its $K$-motive is of unit type. Is it true that the Chow motive of $X$ is of Lefschetz type?}
\end{itemize}

\medskip

This question had been already asked by Bernardara and Tabuada in~\cite{BT}. In higher dimensions, the answer to Question is negative by~\cite[Prop.\,2.3]{BT}, where it is constructed an example of a quadric $X$ over a non-algebraically closed field such that the \mbox{$K$-motive} of~$X$ is of unit type and the Chow motive of $X$ is not of Lefschetz type. According to~\cite[Ex.\,5.4]{BT}, one can take the ground field $\bbQ(t_1,t_2,t_3)$, where $t_1,t_2,t_3$ are independent variables, and let $X$ be the \mbox{six-dimensional} Pfister quadric that corresponds to the quadratic form ${\langle 1,t_1\rangle\otimes \langle 1,t_2\rangle\otimes \langle 1,t_3\rangle}$. Nevertheless, it is not known whether the answer to Question is positive over an algebraically closed field.

Notice that with rational coefficients the Question is simple, because there is a well-known close relation between the categories of rational Chow motives and of rational $K$-motives, see Tabuada's~\cite[Theor.\,1.1]{Tabmot} in the context of non-commutative motives instead of $K$-motives and also the equivalence of categories~\eqref{eq:equivOrlov} in Section~\ref{sect:rat}. This implies that for any smooth projective variety, its rational Chow motive is of Lefschetz type if and only if its rational \mbox{$K$-motive} is of unit type, see~\cite[Theor.\,2.1]{BT},~\cite[Theor.\,1.3]{MT},~\cite[Prop.\,2.1(1)]{GKMS},~\cite[\S\,2]{Vial}, and also Proposition~\ref{prop:known}.

\medskip

The main result of the paper is a positive answer to Question in dimensions less or equal to three, see Theorem~\ref{thm:main} (in the case of dimension three we also assume that the characteristic of the ground field is not two). The most essential part in the proof of the theorem is the case of a threefold, the cases of dimensions one and two are much easier, see Remark~\ref{rmk:lawdim}.

\medskip

A rather direct application of Theorem~\ref{thm:main} is the following result, see Theorem~\ref{thm:main2}. Let $X$ be a smooth projective variety that admits a full exceptional collection and suppose that the dimension of $X$ is less or equal to three (in the case of dimension three we also assume that the characteristic of the ground field is not two); then the Chow motive of $X$ is of Lefschetz type. In fact, we require in Theorem~\ref{thm:main2} a weaker condition on $X$, namely, that $X$ admits only an exceptional collection of expected length.

When $X$ is a curve, the statement of Theorem~\ref{thm:main2} is easy. When $X$ is a surface, this was proved previously by Vial in~\cite[Theor.\,2.7]{Vial} by a different method based on delicate properties of exceptional collections on surfaces obtained in recent papers by Perling~\cite{Per} and Kuznetsov~\cite{KuzMS}.

\medskip

A motivation for Theorem~\ref{thm:main2} is as follows: it seems that varieties with a full exceptional collection tend to have Chow motives of Lefschetz type. The first example of a full exceptional collection was elaborated by Beilinson on a projective space, see~\cite{Bei}. Full exceptional collections are constructed now on many different varieties. Among them one has Grassmanians,~see~\cite{Kap} and~\cite{Fon}, and more general homogenous spaces over algebraically closed fields, see~\cite{KP} and references therein. The Chow motives of these varieties are known to be of Lefschetz type. Let us also mention another interesting example: recently, Kuznetsov has shown in~\cite{Kuz2} that the Chow motive of a certain K\"uchle fivefold (see~\cite{Kuz1} for K\"uchle varieties) is of Lefschetz type and it is expected that this K\"uchle fivefold admits a full exceptional collection. Finally, note that Orlov has constructed in~\cite{Or1},~\cite{Or2} embeddings of arbitrary exceptional collections into derived categories of smooth projective varieties and the Chow motives of these varieties are of Lefschetz type again.

\medskip

The author is grateful to D.\,Orlov and I.\,Panin for discussions on this subject and especially to A.\,Kuznetsov for a careful reading of the text and many useful suggestions.

\section{Statement of the main result}\label{sec:main}

Let $k$ be a field. All varieties are assumed to be over $k$ unless another ground field is specified explicitly. Given a field extension $k\subset L$ and a variety $V$, by $V_L$ we denote the extension of scalars of $V$ from $k$ to $L$.

We refer to~\cite{Ma} for details on the categories of Chow motives and $K$-motives. By~$\CH(k)$ denote the category of Chow motives over $k$ and by $\KM(k)$ denote the category of \mbox{$K$-motives} over $k$. Given a smooth projective variety $V$, by $M(V)$ denote its Chow motive in $\CH(k)$ and by $KM(V)$ denote its $K$-motive in $\KM(k)$. For irreducible smooth projective varieties $V_1$ and~$V_2$, we have
$$
\Hom_{\CH(k)}\big(M(V_1),M(V_2)\big)=CH^d(V_1\times V_2)\,,
$$
$$
\Hom_{\KM(k)}\big(KM(V_1),KM(V_2)\big)=K_0(V_1\times V_2)\,,
$$
where $d$ is the dimension of $V_1$. The assignments ${V\mapsto M(V)}$ and ${V\mapsto KM(V)}$ define contravariant functors from the category of smooth projective varieties over $k$ to the categories~$\CH(k)$ and~$\KM(k)$, respectively. The categories~$\CH(k)$ and $\KM(k)$ have natural symmetric monoidal structures that come from products of varieties. In both categories, the unit object is the motive of the point $\Spec(k)$, which we denote by $\uno$.

There are isomorphisms $M(\PP^1)\simeq\uno\oplus\Le$ and $KM(\PP^1)\simeq \uno\oplus\uno$, where $\Le$ is the Lefschetz motive. For short, we put $\Le^i:=\Le^{\otimes i}$ for $i\in\Z$, where $\Le^{-1}$ is the dual of $\Le$. A Chow motive in $\CH(k)$ is of {\it Lefschetz type} if it is isomorphic to a (finite) direct sum of copies of $\Le^i$ for some integers $i\in \Z$. A $K$-motive in~$\KM(k)$ is of {\it unit type} if it is isomorphic to a (finite) direct sum of copies of $\uno$.

\medskip

\begin{theo}\label{thm:main}
Let $X$ be a smooth projective variety of dimension $d$ over a field $k$. Suppose that the $K$-motive $KM(X)$ is of unit type and one of the following conditions is satisfied:
\begin{itemize}
\item[(i)]
we have $d\leqslant 2$;
\item[(ii)]
we have $d=3$ and the characteristic of $k$ is not $2$.
\end{itemize}
Then the Chow motive $M(X)$ is of Lefschetz type.
\end{theo}

The next sections of the paper consist in the proof of Theorem~\ref{thm:main}.

\medskip

Let us say that a smooth projective variety $V$ admits an {\it exceptional collection of expected length} if $V$ has an exceptional collection $E_1,\ldots,E_n$ such that for any field extension $k\subset L$, the classes of $(E_1)_L,\ldots,(E_n)_L$ generate (freely) the group $K_0(V_L)$.

One often considers the Euler pairing on the group $K_0(V)$, which is defined by the formula
$$
\chi\;:\; K_0(V)\otimes K_0(V) \longrightarrow \Z\,,\qquad \chi([E],[F]):=\sum_{i\geqslant 0}\dim\Ext^i(E,F)=\chi\big(V,{\mathcal H}om(E,F)\big)\,,
$$
where $[E]$ and $[F]$ are classes in $K_0(V)$ of vectors bundles $E$ and $F$ on $V$, respectively, and~$\chi(V,G)$ denotes the Euler characteristic of a vector bundle~$G$ on~$V$. An advantage of the Euler pairing is that it has a categorical meaning being defined in terms of the derived category of coherent sheaves on $V$ only. That is, derived equivalent varieties have the same Euler pairing. However the Euler pairing is neither symmetric, nor antisymmetric.

We will also use the following symmetric pairing:
\begin{equation}\label{eq:inter-pair}
\tau\;:\; K_0(V)\otimes K_0(V) \longrightarrow \Z\,,\qquad \tau([E],[F]):=\chi(V,E\otimes F)\,.
\end{equation}
Notice that the pairing $\tau$ does depend on the choice of the variety $V$ and is not well-defined for the derived category of coherent sheaves on $V$. That is, the pairing $\tau$ does not stay invariant under derived equivalences.

\medskip

Theorem~\ref{thm:main} implies the following result.

\begin{theo}\label{thm:main2}
Let $X$ be a smooth projective variety of dimension $d$ over a field $k$. Suppose that $X$ admits an exceptional collection of expected length and one of the following conditions is satisfied:
\begin{itemize}
\item[(i)]
we have $d\leqslant 2$;
\item[(ii)]
we have $d=3$ and the characteristic of $k$ is not $2$.
\end{itemize}
Then the Chow motive $M(X)$ is of Lefschetz type.
\end{theo}
\begin{proof}
Let us show that the $K$-motive of $X$ is of unit type. For this there are several arguments known to experts and we provide one of them for the sake of completeness (another approaches are, for example, to use a resolution of the structure sheaf of the diagonal on $X\times X$ or, more generally, to use that semi-orthogonal decompositions lead to decompositions of $K$-motives, see~\cite[Sect.\,4]{GO}).

Since $X$ admits an exceptional collection of expected length, $K_0(X)$ is a free abelian group of finite rank and the Euler pairing is unimodular. Indeed, the Euler pairing is given by an upper-triangular matrix with units on the diagonal in the basis in $K_0(X)$ given by the classes of elements in an exceptional collection of expected length. The pairing $\tau$ is obtained from the Euler pairing by applying the duality isomorphism
$$
K_0(X)\longrightarrow K_0(X)\,,\qquad [E]\longmapsto [E^{\vee}]\,,
$$
to the first argument. Therefore the pairing $\tau$ is unimodular as well.

Now let $x_1,\ldots,x_r$ be a basis in $K_0(X)$ and let $y_1,\ldots,y_r$ be the dual basis with respect to the symmetric pairing $\tau$. For each $i$, $1\leqslant i\leqslant n$, define an element $\pi_i:=p_1^*x_i\otimes p_2^*y_i$ in~$K_0(X\times X)$, where $p_1,p_2\colon X\times X\to X$ are the natural projections. One checks easily that $\pi_i$ are orthogonal idempotents in the ring $K_0(X\times X)=\End_{\KM(k)}\big(KM(X)\big)$. Furthermore, the decomposition of identity into a sum of orthogonal idempotents
$$
\mbox{$1=\sum\limits_{i=1}^n\pi_i+\left(1-\sum\limits_{i=1}^n\pi_i\right)$}
$$
defines the decomposition of the $K$-motive $KM(X)$ into a sum of a $K$-motive of unit type and a rest \mbox{$K$-motive}~$Q$:
$$
KM(X)\simeq \uno^{\oplus r}\oplus Q\,.
$$
This decomposition is compatible with scalar extensions with respect to extensions of the field~$k$. Hence, by the definition of an exceptional collection of expected length, for any field extension~$k\subset L$, we have $K_0(Q_L)=0$. One shows that $Q=0$ using the same argument as in the proof of~\cite[Lem.\,5.3]{GO}. Thus $KM(X)$ is of unit type and we apply Theorem~\ref{thm:main}.
\end{proof}

\section{Rational Chow motives of Lefschetz type}\label{sect:rat}

By $\CH(k)_{\bbQ}$ denote the $\bbQ$-linear category of rational Chow motives and by $M(V)_{\bbQ}$ denote the rational Chow motive of a smooth projective variety~$V$. We will use the following almost evident facts on rational Chow motives of Lefschetz type.

\begin{lemma}\label{lem:Chowrat}
Let $V$ be an irreducible smooth projective variety of dimension $d$ such that the rational Chow motive $M(V)_{\bbQ}$ in $\CH(k)_{\bbQ}$ is of Lefschetz type. Then the following holds true:
\begin{itemize}
\item[(i)]
for each $i$, $0\leqslant i\leqslant d$, the intersection pairing $CH^i(V)_{\bbQ}\otimes CH^{d-i}(V)_{\bbQ}\to\bbQ$ is non-degenerate; in particular, $CH^i(V)_{\bbQ}$ and $CH^{d-i}(V)_{\bbQ}$ have the same (finite) dimension over $\bbQ$;
\item[(ii)]
for any field extension $k\subset L$ and for each $i$, $0\leqslant i\leqslant d$, the natural homomorphism $CH^i(V)_{\bbQ}\to CH^i(V_L)_{\bbQ}$ is an isomorphism;
\item[(iii)]
for any field extension $k\subset L$, the variety $V_L$ over $L$ is irreducible.
\end{itemize}
\end{lemma}
\begin{proof}
Part~$(i)$ follows from~\cite[Lem.\,2.1]{GO}, where one uses the duality $M(V)^{\vee}_{\bbQ}\simeq M(V)_{\bbQ}\otimes\Le^{-d}$. To show~$(ii)$ one uses that scalar extension is well-defined for Chow motives. Finally,~$(iii)$ is implied by~$(ii)$ with $i=0$, because $CH^0(V_L)_{\bbQ}$ is the $\bbQ$-vector space generated by irreducible components of~$V_L$ and $CH^0(V)_{\bbQ}\simeq \bbQ$ because $V$ is irreducible.
\end{proof}

\medskip

By $\KM(k)_{\bbQ}$ denote the $\bbQ$-linear category of rational $K$-motives and by $KM(V)_{\bbQ}$ denote the rational $K$-motive of a smooth projective variety~$V$. The categories $\CH(k)_{\bbQ}$ and $\KM(k)_{\bbQ}$ are related as follows.

Let $\widetilde{\CH}(k)_{\bbQ}$ be the symmetric monoidal category, where objects are the same as in $\CH(k)_{\bbQ}$ and morphisms are defined by the formula
$$
{\Hom}_{\widetilde{\CH}(k)_{\bbQ}}(M,N):=\bigoplus\limits_{i\in\Z}\,\Hom_{\CH(k)_{\bbQ}}(M,N\otimes \Le^i)
$$
for all rational Chow motives $M$ and $N$ (we do not consider the grading on the right hand side), cf.~\cite[\S\,7]{Tabmot}. Then one has an equivalence of symmetric monoidal categories
\begin{equation}\label{eq:equivOrlov}
{\KM(k)_{\bbQ}\stackrel{\sim}\longrightarrow \widetilde{\CH}(k)_{\bbQ}}\,,
\end{equation}
which was constructed essentially by Orlov in~\cite{Ormot}. For any smooth projective variety $V$,
the equivalence sends its rational $K$-motive $KM(V)_{\bbQ}$ to its rational Chow motive $M(V)_{\bbQ}$. Given irreducible smooth projective varieties $V_1$ and $V_2$, the equivalence gives the map on morphisms
\begin{multline*}
\Hom_{\KM(k)_{\bbQ}}\big(KM(V_1)_{\bbQ},KM(V_2)_{\bbQ}\big)=K_0(V_1\times V_2)_{\bbQ}\stackrel{\sim}\longrightarrow
 \\
\stackrel{\sim}\longrightarrow
\bigoplus\limits_{i\in\Z}CH^i(V_1\times V_2)_{\bbQ}=\Hom_{\widetilde{\CH}(k)_{\bbQ}}\big(M(V_1)_{\bbQ},M(V_2)_{\bbQ}\big)
\end{multline*}
defined by the formula
\begin{equation}\label{eq:Orlov}
\alpha\longmapsto p_1^*\sqrt{{\rm Td}_{V_1}}\cdot {\rm ch}(\alpha)\cdot p_2^*\sqrt{{\rm Td}_{V_2}}\,,
\end{equation}
where $p_1\colon V_1\times V_2\to V_1$, $p_2\colon V_1\times V_2\to V_2$ are the natural projections, $\rm ch(\alpha)$ is the Chern character of~$\alpha$, and ${\rm Td}_{V_1}$, ${\rm Td}_{V_2}$ are the Todd classes of $V_1$, $V_2$, respectively. Grothendieck--Riemann--Roch theorem implies that this definition is correct, that is, respects compositions of morphisms. (Alternatively, following Tabuada in~\cite[\S\,8]{Tabmot}, one can send~$\alpha$ to~${{\rm ch}(\alpha)\cdot p_2^*\,{\rm Td}_{V_2}}$ instead of the right hand side of formula~\eqref{eq:Orlov}, or, more generally, one can send~$\alpha$ to~${(p_1^*\,{\rm Td}_{V_1})^u\cdot{\rm ch}(\alpha)\cdot (p_2^*\,{\rm Td}_{V_2})^{1-u}}$ for any $u\in\bbQ$.)

\medskip

The following result is proved in the context of non-commutative motives in~\cite[Theor.\,2.1]{BT}.

\begin{prop}\label{prop:known}
For any smooth projective variety $V$, the rational \mbox{$K$-motive}~$KM(V)_{\bbQ}$  in~$\KM(k)_{\bbQ}$ is of unit type if and only if the rational Chow motive~$M(V)_{\bbQ}$ in~$\CH(k)_{\bbQ}$ is of Lefschetz type.
\end{prop}

There is a version of Proposition~\ref{prop:known} which asserts that if $V$ admits a full exceptional collection, then the rational Chow motive $M(V)_{\bbQ}$ is of Lefschetz type. For a while this had been a well-known folklore and then different proofs were proposed by  Marcolli and Tabuada in~\cite[Theor.\,1.3]{MT}, by Galkin, Katzarkov, Mellit, and Shinder in~\cite[Prop.\,2.1(1)]{GKMS}, and by Vial in~\cite[\S\,2]{Vial}. An essential part in all these proofs is the Chern character isomorphism (which reveals rationality of coefficients). Any of the proofs cited above can be easily adopted to show Proposition~\ref{prop:known}.

For the sake of completeness, we provide the following argument that proves Proposition~\ref{prop:known}, without claiming any originality. One checks easily that a rational Chow motive $M$ is of Lefschetz type, that is, there is an isomorphism $M\simeq (\Le^{i_1})^{\oplus r_1}\oplus\ldots\oplus(\Le^{i_n})^{\oplus r_n}$ in $\CH(k)_{\bbQ}$ if and only if there is an isomorphism $M\simeq\uno^{\oplus r}$ in $\widetilde{\CH}(k)_{\bbQ}$. Thus the proposition follows directly from the equivalence of categories~\eqref{eq:equivOrlov}.

\section{Splitting off a Lefschetz type motive}

In this section, we decompose the Chow motive of $X$ as in Theorem~\ref{thm:main} into a direct sum of a Chow motive of Lefschetz type and a Chow motive $N$ whose only potentially non-trivial Chow group is~$CH^3(N)$, which is necessarily $2$-torsion (and the same holds for the Chow motive~$N_L$ for any field extension $k\subset L$). We finish the proof of Theorem~\ref{thm:main} in Section~\ref{sec:absence} by showing that $N=0$.

\medskip

We will use several auxiliary results. First, we provide some elementary facts about pairings on filtered abelian groups. Let $A$ be a free finitely generated abelian group with a structure of a commutative ring and a linear map
$$
\chi\;:\; A\longrightarrow\Z
$$
such that the symmetric pairing
$$
\langle\cdot,\cdot\rangle\;:\; A\otimes A\longrightarrow \Z\,,\qquad x\otimes y\longmapsto\chi(x\cdot y)\,,
$$
is unimodular. Let
$$
A=F^0A\supset F^1A\supset\ldots\supset F^dA\supset F^{d+1}A=0
$$
be a multiplicative decreasing filtration, that is, for all $i,j\geqslant 0$, we have ${F^iA\cdot F^jA\subset F^{i+j}A}$. We assume that for each $i$, $0\leqslant i\leqslant d$, there is an equality between the ranks of adjoint quotients:
$$
\rk(\gr^i_FA)=\rk(\gr^{d-i}_FA)\,.
$$
We keep these assumptions on the ring $A$, the form $\chi$, and the filtration $F^{\bullet}A$ during all this section.

\begin{lemma}\label{lem:unimod}
Suppose the filtration $F^{\bullet}A$ splits, that is, for each $i$, $0\leqslant i\leqslant d$, the quotient~$A/F^iA$ is torsion-free. Then for each $i$, $0\leqslant i\leqslant d$, the induced pairing between free finitely generated abelian groups
$$\langle\cdot,\cdot\rangle_i\;:\; \gr_F^iA\otimes\gr^{d-i}_FA\longrightarrow \Z
$$
is unimodular.
\end{lemma}
\begin{proof}
The proof is by induction on $d$. Depending on the parity of $d$, the base of the induction is either the case $A=0$ for $d$ odd, or the case $F^1A=0$ for $d$ even. In both cases, the assertion is clear.

Let us make an induction step. Since the pairing $\langle\cdot,\cdot\rangle$ is unimodular, it induces an isomorphism $A\stackrel{\sim}\longrightarrow A^{\vee}$. Since the filtration splits, the natural homomorphism
$$
\phi\;:\; A^{\vee}\longrightarrow (F^dA)^{\vee}=(\gr^d_FA)^{\vee}
$$
is surjective. The vanishing ${\langle F^1A,F^dA\rangle=0}$ implies that $\phi$ factors through the quotient
$$
A\longrightarrow A/F^1A=\gr^0_FA\,.
$$
One checks easily that the arising surjection $\psi\colon\gr_F^0A\twoheadrightarrow (\gr^d_FA)^{\vee}$ is induced by the pairing ${\langle\cdot,\cdot\rangle_0\colon\gr_F^0A\otimes\gr_F^dA\to \Z}$. Since $\gr^0_FA$ and~$\gr^d_FA$  are free finitely generated abelian groups of the same rank, we see that $\psi$ is an isomorphism, whence the pairing $\langle\cdot,\cdot\rangle_0$ is unimodular. This implies also the equality $(F^dA)^{\bot}=F^1A$. Usng this and the fact that the filtration splits, we see that the isomorphism $A\stackrel{\sim}\longrightarrow A^{\vee}$ defines the isomorphisms
$$
F^1A\stackrel{\sim}\longrightarrow (A/F^dA)^{\vee}\,,\qquad F^1A/F^dA\stackrel{\sim}\longrightarrow (F^1A/F^dA)^{\vee}\,.
$$
Thus the induced pairing $F^1A/F^dA\otimes F^1A/F^dA\to \Z$ is unimodular and we complete the induction step replacing $A$ by $F^1A/F^dA$ and decreasing $d$ by $2$.
\end{proof}

\begin{remark}
\hspace{0cm}
\begin{itemize}
\item[(i)]
Actually, Lemma~\ref{lem:unimod} and its proof do not involve the ring structure on $A$ and are valid in a more general case. Namely, one can replace the ring structure on $A$, the linear map~$\chi$, and the multiplicative property of the filtration $F^{\bullet}A$ by the following: $\langle\cdot,\cdot\rangle$ is any unimodular pairing on~$A$ and the filtration $F^{\bullet}A$ satisfies the condition $\langle F^iA,F^jA\rangle=0$ for all $i,j\geqslant 0$ with $i+j\geqslant d+1$.
\item[(ii)]
A more direct but less invariant proof of Lemma~\ref{lem:unimod} is to choose a splitting of the filtration and to choose a basis of $A$ by choosing bases of all adjoint quotients. In this basis of~$A$, the Gram matrix $G$ of the pairing $\langle\cdot,\cdot\rangle$ has the form
$$
G=\left(\begin{array}{ccccc}
* & *&\dots & * & G_0 \\
* & & & G_1& 0\\
\vdots && \iddots & & \vdots\\
* & G_{d-1} &&  & 0\\
G_d & 0 & \dots & 0&0
\end{array}\right)\,,
$$
where $G_i$ is the matrix of the pairing $\langle\cdot,\cdot\rangle_i$ for each $i$, $0\leqslant i\leqslant d$. The equality between ranks implies that each $G_i$ is a square matrix. Hence $\det(G)$ equals up to sign to the product $\det(G_0)\cdot\ldots\cdot\det(G_d)$, which proves the lemma.
\end{itemize}
\end{remark}

\begin{lemma}\label{lemma:secondkey}
Suppose that $\rk(\gr^0_FA)=1$ (and so $\rk(\gr^d_FA)=\rk(F^dA)=1$ as well) and that there is $i_0$, $0\leqslant i_0\leqslant d$, such that the quotients $A/F^{i_0}A$ and $A/F^{d-i_0}A$ are torsion-free. Then $\chi\colon F^dA\to \Z$ is an isomorphism and the quotient $A/F^dA$ is torsion-free.
\end{lemma}
\begin{proof}
Modify the filtration $F^{\bullet}A$ as follows: for each $i$, $0\leqslant i\leqslant d$, let $\widetilde{F}^iA\subset A$ be the saturation of $F^iA$ in~$A$, that is, $\widetilde{F}^iA$ is the preimage of the torsion subgroup under the quotient map $A\to A/F^iA$. Clearly, the filtration $\widetilde{F}^{\bullet}A$ is multiplicative and for each $i$, $0\leqslant i\leqslant d$, there are equalities
$$
\rk(\gr^i_{\widetilde{F}}A)=\rk(\gr^i_FA)=\rk(\gr^{d-i}_FA)=\rk(\gr^{d-i}_{\widetilde{F}}A)\,.
$$
Also, by construction, the filtration $\widetilde{F}^{\bullet}A$ splits. Thus by Lemma~\ref{lem:unimod}, for each $i$, $0\leqslant i\leqslant d$, the pairing ${\widetilde{\langle\cdot,\cdot\rangle}_{i}\colon\gr^{i}_{\widetilde{F}}A\otimes \gr^{d-i}_{\widetilde{F}}A\to \Z}$
is unimodular. In particular, this map is surjective.

Since the quotients $A/F^{i_0}A$ and $A/F^{d-i_0}A$ are torsion-free, there are equalities ${F^{i_0}A=\widetilde{F}^{i_0}A}$ and ${F^{d-i_0}A=\widetilde{F}^{d-i_0}A}$. Hence the natural maps
$$
F^{i_0}A\longrightarrow \gr^{i_0}_{\widetilde{F}}A\,,\qquad F^{d-i_0}A\longrightarrow \gr^{d-i_0}_{\widetilde{F}}A
$$
are surjective. The composition of surjective maps
$$
F^{i_0}A\otimes F^{d-i_0}A\longrightarrow \gr^{i_0}_{\widetilde{F}}A\otimes \gr^{d-i_0}_{\widetilde{F}}A\stackrel{\widetilde{\langle\cdot,\cdot\rangle}_{i_0}}\longrightarrow \Z
$$
also factors as the composition
$$
F^{i_0}A\otimes F^{d-i_0}A\longrightarrow F^dA\stackrel{\chi}\longrightarrow \Z\,.
$$
This implies that the map $\chi\colon F^dA\to\Z$ is surjective. Since the rank of $F^dA$ is one and~$F^dA$ is torsion-free, we see that this map is an isomorphism. This implies directly that the quotient~$A/F^dA\simeq\Ker(\chi)$ is torsion-free.
\end{proof}

\begin{cor}\label{cor:key}
Suppose that $\gr^0_FA=A/F^1A\simeq \Z$, the adjoint quotient $\gr_F^1A$ is torsion-free, and that $d=3$.
Then $\chi\colon F^3A\to \Z$ is an isomorphism, the adjoint quotients $\gr_F^1A$ and $\gr^2_FA$ are torsion-free, and the pairing ${\langle\cdot,\cdot\rangle_1\colon \gr_F^1A\otimes\gr^2_FA\to \Z}$ is unimodular.
\end{cor}
\begin{proof}
The quotient $A/F^{2}A$ is torsion-free being an extension of $\gr^0_FA\simeq \Z$ by the torsion-free group $\gr^1_FA$. Hence by Lemma~\ref{lemma:secondkey} with $i_0=1$, we have that $\chi\colon F^3A\to \Z$ is an isomorphism and the quotient~$A/F^3A$ is torsion-free. Thus the filtration $F^{\bullet}A$ splits and we conclude the proof using Lemma~\ref{lem:unimod}.
\end{proof}

\begin{remark}\label{rmk:2d}
Suppose that $\gr^0_FA=A/F^1A\simeq \Z$ and $d=2$. A similar argument as in the proof of Corollary~\ref{cor:key} implies that $\chi\colon F^2A\to \Z$ is an isomorphism, the adjoint quotient $\gr^1_FA$ is torsion-free, and the pairing ${\langle\cdot,\cdot\rangle_1\colon \gr_F^1A\otimes\gr^1_FA\to \Z}$ is unimodular.
\end{remark}

\medskip

Let $V$ be a smooth projective variety of dimension~$d$. We apply the above results with~$A$ being the ring $K_0(V)$ and $\chi$ being the Euler characteristic, so that $\langle\cdot,\cdot\rangle$ is the pairing $\tau$ (see formula~\eqref{eq:inter-pair}). Let $F^iK_0(V)$, $i\geqslant 0$, be the filtration on $K_0(V)$ by codimension of support, that is, $F^iK_0(V)$ is generated by classes of coherent sheaves whose support has codimension at least~$i$. Recall the following important facts, see, e.g.,~\cite[Exp.\,0]{BGI} or \cite[Ex.\,15.1.5, Ex.\,15.3.6]{Ful}.

\begin{prop}\label{prop:surjvarphi}
For each $i\geqslant 0$, there is a surjective homomorphism
$$
\varphi_i\;:\; CH^i(V)\twoheadrightarrow\gr_F^iK_0(V)
$$
that sends the class of an irreducible subvariety $Z\subset V$ of codimension $i$ to the class of its structure sheaf~$\cO_Z$. The homomorphism $\varphi_i$ commutes in a natural sense with scalar extensions with respect to extensions of the field $k$. The kernel of the homomorphism~$\varphi_i$ is contained in~$(i-1)!$-torsion of~$CH^i(V)$. In particular, $\varphi_i$ is an isomorphism for $i=0,1,2$ and the kernel of $\varphi_3$ is contained in \mbox{$2$-torsion} of $CH^3(V)$.
\end{prop}

Let us mention that for each $i\geqslant 0$, there is also a Chern class map $c_i\colon \gr^i_FK_0(V)\to CH^i(V)$ and we have relations ${c_i\circ\varphi_i=\varphi_i\circ c_i=(-1)^{i-1}(i-1)!}$. In particular, this implies that the kernel of $\varphi_i$ is contained in $(i-1)!$-torsion of $CH^i(V)$.

The filtration $F^{\bullet}K_0(V)$ is multiplicative. For each $i$, $0\leqslant i\leqslant d$, the pairing $\tau$ on~$K_0(V)$ induces a pairing ${\tau_i\colon \gr^i_FK_0(V)\otimes \gr^{d-i}_FK_0(V)\to \Z}$. The composition of the map ${\varphi_i\otimes\varphi_{d-i}}$ with $\tau_i$ equals the intersection pairing between Chow groups. In particular, the composition ${CH^d(X)\stackrel{\varphi_d}\longrightarrow F^dK_0(X)\stackrel{\chi}\longrightarrow\Z}$ equals the degree of zero-cycles.

\medskip

The following statement, as well as its proof, is an analogue of~\cite[Lem.\,2.1]{GO} and a general version of this fact had been proved by Panin in~\cite[Lem.\,7.4]{Pan}. We provide a proof for convenience of the reader.

\begin{lemma}\label{lemma:unimK}
Let $V$ be a smooth projective variety such that the \mbox{$K$-motive}~$KM(V)$ is of unit type. Then $K_0(V)$ is a free finitely generated abelian group and the pairing $\tau$ on $K_0(V)$ is unimodular.
\end{lemma}
\begin{proof}
Since $K_0$-groups are well-defined for $K$-motives, there are isomorphisms
$$
{K_0(V)\simeq K_0\big(KM(V)\big)\simeq K_0(\uno)^{\oplus r}}\simeq\Z^{\oplus r}\,,
$$
where $KM(V)\simeq \uno^{\oplus r}$. Moreover, the functor $K_0$ provides a symmetric monoidal equivalence of symmetric monoidal categories between the category of $K$-motives of unit type and the category of free finitely generated abelian groups (notice that the functor $K_0$ is not monoidal on the whole category $\KM(k)$).

On the other hand, the $K$-motive $KM(V)$ is canonically self-dual in $\KM(k)$. The corresponding evaluation morphism equals the composition
$$
KM(V)\otimes KM(V)\longrightarrow KM(V)\longrightarrow \uno\,,
$$
where the first morphism is given by the pull-back with respect to the diagonal embedding $V\hookrightarrow V\times V$ and the second morphism is given by the class $[\cO_V]\in K_0(V)=K_0\big(V\times\Spec(k)\big)$. One checks directly that the functor $K_0$ sends this evaluation morphism to the pairing~$\tau$.

We obtain that the pairing $\tau$ provides a self-duality of the group $K_0(V)$, that is, $\tau$ is unimodular.
\end{proof}

The following useful result is proved in~\cite[Lem.\,2.2]{GKMS}.

\begin{prop}\label{prop:galkin}
Let $V$ be a smooth projective variety such that the group $K_0(V)$ is torsion-free. Then the group~$CH^1(V)\simeq\gr^1_FK_0(V)$ is torsion-free as well.
\end{prop}

\medskip

Now consider the case of a threefold with $K$-motive of unit type. Given an abelian group~$\Gamma$ and a natural number $l$, by $\Gamma_l$ denote the $l$-torsion subgroup of $\Gamma$.

\begin{prop}\label{prop:key}
Let $X$ be an irreducible smooth projective variety of dimension $3$ such that the \mbox{$K$-motive}~$KM(X)$ is of unit type. Then the following holds true:
\begin{itemize}
\item[(i)]
the degree map gives an isomorphism
$$
{\rm deg}\;:\; CH^3(X)/CH^3(X)_2\stackrel{\sim}\longrightarrow \Z\,;
$$
in particular, $X$ has a zero-cycle of degree $1$;
\item[(ii)]
the Chow groups $CH^1(X)$ and $CH^2(X)$ are free finitely generated abelian groups and the intersection pairing
$$
\qquad CH^1(X)\otimes CH^{2}(X)\longrightarrow \Z
$$
is unimodular;
\item[(iii)]
for any field extension $k\subset L$ and for $i=0,1,2$, the natural homomorphism
$$
CH^i(X)\longrightarrow CH^i(X_L)
$$
is an isomorphism.
\end{itemize}
\end{prop}
\begin{proof}
By Lemma~\ref{lemma:unimK}, $K_0(X)$ is a free finitely generated abelian group and the pairing $\tau$ on~$K_0(X)$ is unimodular. By Proposition~\ref{prop:known}, the rational Chow motive of $X$ is of Lefschetz type. Hence by Lemma~\ref{lem:Chowrat}$(i)$ and Proposition~\ref{prop:surjvarphi}, for each~$i$, $0\leqslant i\leqslant d$, there are equalities
$$
\rk\big(\gr^i_FK_0(X)\big)=\rk\big(CH^i(X)\big)=\rk\big(CH^{d-i}(X)\big)=\rk\big(\gr^{d-i}_FK_0(X)\big)\,.
$$
Since $X$ is irreducible, there is an isomorphism $\Z\simeq CH^0(X)\simeq\gr^0_FK_0(X)$. By Proposition~\ref{prop:galkin}, $CH^1(X)\simeq \gr^1_FK_0(X)$ is torsion-free. Hence by Corollary~\ref{cor:key} and Proposition~\ref{prop:surjvarphi}, we obtain~$(i)$ and~$(ii)$. Also, we see that the filtration $F^{\bullet}K_0(X)$ splits.

Consider a field extension $k\subset L$. Using Lemma~\ref{lem:Chowrat}$(ii)$ and the fact that $(\varphi_i)_{\bbQ}$ are isomorphisms and commute with extension of scalars by Proposition~\ref{prop:surjvarphi}, we see that the arising morphism of filtered groups $\eta\colon F^{\bullet}K_0(X)\to F^{\bullet}K_0(X_L)$ is an isomorphism after tensoring with~$\bbQ$. In addition, the homomorphism $K_0(X)\to K_0(X_L)$ is an isomorphism, because $KM(X)$ is of unit type. As it was shown above, the filtration $F^{\bullet}K_0(X)$ splits. Altogether this implies that the morphism of filtered groups $\eta$ is an isomorphism. By Proposition~\ref{prop:surjvarphi}, this proves~$(iii)$.
\end{proof}

\begin{remark}
A similar argument as in the proof of Proposition~\ref{prop:key}$(i)$ shows that if $X$ is a smooth projective variety $X$ of dimension $4$ with $K$-motive of unit type, then the degree map gives an isomorphism ${\rm deg}\colon CH^4(X)/CH^4(X)_6\stackrel{\sim}\longrightarrow \Z$. Namely, using Proposition~\ref{prop:galkin}, one applies Lemma~\ref{lemma:secondkey} with $i_0=2$.
\end{remark}

\medskip

The following standard argument shows that Proposition~\ref{prop:key} allows to split a Lefschetz type motive out of $M(X)$.

\begin{cor}\label{cor:split}
Under assumptions of Proposition~\ref{prop:key}, there is an isomorphism of Chow motives
$$
M(X)\simeq M\oplus N\,,
$$
where $M$ is of Lefschetz type and $N$ is such that for any field extension $k\subset L$, we have $CH^i(N_L)=0$ for $i=0,1,2$ and $CH^3(N_L)$ coincides with~$CH^3(X_L)_2$.
\end{cor}
\begin{proof}
By Proposition~\ref{prop:key}$(i)$, there is a zero-cycle $\alpha\in CH^3(X)$ of degree $1$. By Proposition~\ref{prop:key}$(ii)$, we may choose a basis $D_1,\ldots,D_r$ in the free abelian group $CH^1(X)$, and the dual basis $C_1,\ldots,C_r$ in the free abelian group~$CH^2(X)$. Given elements $a\in CH^i(X)$ and $b\in CH^j(X)$, put
$$
a\times b:=p_1^*a\cdot p_2^*b\in CH^{i+j}(X\times X)\,,
$$
where $p_1,p_2\colon X\times X\to X$ are the natural projections. Define an element
$$
\mbox{$\pi:=X\times\alpha+\sum\limits_{i=1}^rD_i\times C_i+\sum\limits_{i=1}^rC_i\times D_i+\alpha\times X$}\in CH^3(X\times X)\,.
$$
One checks easily that $\pi$ is an idempotent as a correspondence from $X$ to itself. Let $N$ be a Chow motive that splits out of $M(X)$ by the idempotent~${1-\pi}$. We obtain a decomposition
$$
M(X)\simeq \uno\oplus\Le^{\oplus r}\oplus (\Le^2)^{\oplus r}\oplus\Le^3\oplus N\,.
$$
Note that the natural homomorphism
$$
CH^3(X)/CH^3(X)_2\to CH^3(X_L)/CH^3(X_L)_2
$$
is an isomorphism, being the identity map from $\Z$ to itself by Proposition~\ref{prop:key}$(i)$. Thus Proposition~\ref{prop:key}$(iii)$ implies that the motive $N$ satisfies all conditions claimed in the corollary.
\end{proof}

\begin{remark}\label{rmk:Chowvan}
Let $P$ be a Chow motive over $k$ such that for any field extension $k\subset L$, all Chow groups of the Chow motive $P_L$ vanish. Using the same argument as in the proof of~\cite[Lem.\,1]{GG}, one concludes that~${P=0}$. Thus in order to prove Theorem~\ref{thm:main} with the help of Corollary~\ref{cor:split}, it remains to show that $CH^3(X_L)_2=0$ for any field extension $k\subset L$.
\end{remark}

\begin{remark}\label{rmk:lawdim}
The same arguments as above apply in dimensions $1$ and $2$ as well. Namely, we have the following reasonings.
\begin{itemize}
\item[(i)]
Let $X$ be an irreducible smooth projective curve such that the $K$-motive $KM(X)$ is of unit type. The filtration $F^{\bullet}K_0(X)$ splits, because $\gr^0_FK_0(X)\simeq \Z$. By Lemma~\ref{lem:unimod} with $d=1$, we see that the pairing $CH^0(X)\otimes CH^1(X)\to \Z$ is unimodular. This implies that the Jacobian of $X$ vanishes and that $X$ has a zero-cycle of degree $1$. Thus $X\simeq \PP^1$ and henceforth the Chow motive of $X$ is of Lefschetz type.
\item[(ii)]
Let $X$ be an irreducible smooth projective surface such that the $K$-motive $KM(X)$ is of unit type. One easily modifies the argument in the proof of Proposition~\ref{prop:key} replacing Corollary~\ref{cor:key} by Remark~\ref{rmk:2d} (in this case, one does not use Proposition~\ref{prop:galkin}). Thus one obtains that there is an isomorphism $CH^2(X)\simeq\Z$, the Chow group $CH^1(X)$ is a free finitely generated group, the intersection pairing $CH^1(X)\otimes CH^1(X)\to\Z$ is unimodular, and for any field extension $k\subset L$ and for $i=0,1,2$, the natural homomorphism $CH^i(X)\longrightarrow CH^i(X_L)$ is an isomorphism. Then a splitting argument as in the proof of Corollary~\ref{cor:split} together with Remark~\ref{rmk:Chowvan} imply directly that the Chow motive of $X$ is of Lefschetz type.
\end{itemize}
\end{remark}

\section{Absence of torsion zero-cycles}\label{sec:absence}

In this section, we show that the group~$CH^3(X)$ is torsion-free for any threefold $X$ as in Theorem~\ref{thm:main}. Together with results of the previous section, this allows us to prove the theorem.

Here is the plan of the proof of the vanishing of torsion in $CH^3(X)$. It follows from Proposition~\ref{prop:surjvarphi} that $CH^3(X)$ has only $2$-torsion and it is equal to the kernel of~$\varphi_3$. This kernel coincides with the image of the differential $d_2\colon H^1(X,\cK_2)\to CH^3(X)$ in the Brown--Gersten spectral sequence, which factors through the quotient $H^1(X,\cK_2)/2$. It turns out that the product map
$$
\mu\;:\;k^*/2\otimes CH^1(X)/2\longrightarrow H^1(X,\cK_2)/2
$$
is an isomorphism, which implies that $d_2$ vanishes. To show that $\mu$ is an isomorphism, first we use that by the Merkurjev--Suslin theorem, one has an embedding
$$
H^1(X,\cK_2)/2\hookrightarrow H^3_{\acute e t}(X,\Z/2)\,.
$$
For simplicity, assume for a moment that the characteristic of $k$ is zero. The Hochschild--Serre spectral sequence expresses \'etale cohomology of $X$ in terms of \'etale cohomology of the scalar extension $X_{\bar k}$ of $X$ to the algebraic closure $\bar k$ of $k$. By the Lefschetz principle, we can assume that $\bar k\subset \C$, so that \'etale cohomology of $X_{\bar k}$ with coefficients in $\Z/2$ can be computed in terms classical complex cohomology of $X(\C)$ with coefficients in $\Z$. Now the Atiyah--Hirzeburch spectral sequence relates integral cohomology of $X(\C)$ with topological \mbox{$K$-groups} of $X(\C)$. Note that the Atiyah--Hirzeburch spectral sequence degenerates for smooth projective complex threefolds. Finally, since the $K$-motive of $X$ is of unit type, topological $K$-groups of $X(\C)$ have a very simple structure, which allows to work effectively with the Atiyah--Hirzebruch spectral sequence.

\medskip

Now let us fulfil this plan. We start by analyzing cohomology of a threefold whose \mbox{$K$-motive} is of unit type when the ground field is either $\C$, or, more generally, is separably closed. For complex varieties, we will use topological $K$-theory, while, for varieties over an arbitrary separably closed field, we will use \'etale $K$-theory developed by Friedlander in~\cite{Fr1},~\cite{Fr2}, Dwyer and Friedlander in~\cite{DF}, and Thomason in~\cite{Th} (see also a survey in~\cite[\S\,1.5]{Gei}).

\medskip

Let $V$ be a smooth algebraic variety over a separably closed field $L$. Let $l$ be a prime number different from the characteristic of $L$ and choose a generator of the Tate module $\Z_l(1)$, that is, a compatible system of $l$-primary roots of unity in $L$. The choice allows us to ignore Tate twists in \'etale cohomology.

By $H^i(V)$, $i\geqslant 0$, denote either the classical complex cohomology group $H^i\big(V(\C),\Z\big)$ when $L=\C$, or the \'etale cohomology group $H^i_{\acute e t}(V,\Z_l)$ when $L$ is an arbitrary separably closed field. Respectively, by~$KT_i(V)$, $i\geqslant 0$, denote either the topological $K$-group $K_i^{top}\big(V(\C)\big)$, or the \'etale $K$-group~$\hat{K}_i^{\acute e t}(V)$ related with the prime number~$l$. Namely, let $U(\infty)$ be the infinite unitary group,~$BU(\infty)$ be its classifying space, and let $\Map$ denote the space of continuous maps between topological spaces (or simplicial sets, pro-simplicial sets, etc). By $\Map\big(V,\Z\times BU(\infty)\big)$ denote either the space~${\Map\big(V(\C),\Z\times BU(\infty)\big)}$, or the space~${\Map\big(V_{\acute e t},(\Z\times BU(\infty))\,\hat{}_l\,\big)}$, where  a pro-simplicial set~$V_{\acute e t}$ is the \'etale homotopy type of $V$ defined by Artin and Mazur in~\cite{AM} and~$\big(\Z\times BU(\infty)\big)\,\hat{}_l$ is an $l$-adic completion of the space~${\Z\times BU(\infty)}$ defined in an appropriate way, see more detail in~\cite[\S\,1]{Fr1}. Then there are isomorphisms
$$
KT_i(V)\simeq\pi_i \Map\big(V,\Z\times BU(\infty)\big)\,,\qquad i\geqslant 0\,.
$$
The topological $K$-group $K_0^{top}\big(V(\C)\big)$ is the Grothendieck group of complex vector bundles on the manifold~$V(\C)$.

Note that the loop space $\Omega\big(\Z\times BU(\infty)\big)$ is homotopy equivalent to $U(\infty)$ and, by Bott periodicity, $\Omega U(\infty)$ is homotopy equivalent to $\Z\times BU(\infty)$. Hence there are isomorphisms ${KT_{2i}(V)\simeq KT_0(V)}$ and ${KT_{2i+1}(V)\simeq KT_1(V)}$, $i\geqslant 0$.

\medskip
One has the Atiyah--Hirzebruch spectral sequence
$$
E_2^{ij}=H^i(V,j/2)\Rightarrow KT_{-i-j}(V)\,,
$$
where $H^i(V,j/2):=H^i(V)$ for $j$ even and $H^i(V,j/2)=0$ for $j$ odd. The differential $d_2$ is zero and the $E_3$-term looks as follows:
$$
\xymatrix{
0&0&0&0&0&0&0\\
H^0(V)\ar[rrrdd]^{d_3}&H^1(V)\ar[rrrdd]^{d_3}&H^2(V)\ar[rrrdd]^{d_3}&H^3(V)\ar[rrrdd]^{d_3}&H^4(V)&H^5(V)&H^6(V)\\
0&0&0&0&0&0&0\\
H^0(V)&H^1(V)&H^2(V)&H^3(V)&H^4(V)&H^5(V)&H^6(V)}
$$
The spectral sequence is periodic with respect to the vertical shift by two and degenerates in the $E_2$-term after tensoring with~$\bbQ$. In particular, the images of all differentials are torsion groups.

Originally, the Atiyah--Hirzebruch spectral sequence was constructed in~\cite{AH0} by taking the increasing skeletal filtration on $V(\C)$ (when $L=\C$). One obtains the same spectral sequence by taking the decreasing Postnikov filtration on $\Z\times BU(\infty)$ (see, e.g.,~\cite[Theor.\,B.8]{GM} for the equivalence of these two approaches). Note that the $j$-th step of the Postnikov filtration on~$\Z\times BU(\infty)$ is homotopy equivalent to the Eilenberg--Maclane space $K(\Z,j)$ for $j$ even and is trivial for $j$ odd and there are natural isomorphisms
$$
\pi_i\Map\big(V,K(\Z,j)\big)\simeq H^{j-i}(V)\,,\qquad i\leqslant j\,,
$$
$$
\pi_i\Map\big(V,K(\Z,j)\big)=0\,,\qquad \, i>j\,.
$$
The first two non-trivial steps of the Postnikov filtration on $\Z\times BU(\infty)$ look as follows:
$$
\Z\times BU(\infty)\longrightarrow \Z\sim K(\Z,0)\,,
$$
$$
BU(\infty)\longrightarrow BU(1)\sim K(\Z,2)\,,
$$
where the first map is the natural projection, the second map is induced by the determinant~${U(\infty)\to U(1)}$, and $\sim$ denotes homotopy equivalence. Clearly, both maps have splittings (for the second map one uses the embedding $U(1)\to U(\infty)$ defined by any diagonal entry).

It follows that the Atiyah--Hirzebruch spectral sequence defines surjective homomorphisms
$$
KT_0(V)\to H^0(V)\,,\qquad KT_1(V)\longrightarrow H^1(V)\,,\qquad \widetilde{K}T_0(V)\longrightarrow H^2(V)\,,
$$
where $\widetilde{K}T_0(V)$ denotes the kernel of the homomorphism $KT_0(V)\to H^0(V)$. We conclude that all differentials in the Atiyah--Hirzebruch spectral sequence that come out of~$H^0(V)$, $H^1(V)$, and $H^2(V)$ are zero.

When $L=\C$, the homomorphism $K_0^{top}\big(V(\C)\big)\to H^0\big(V(\C),\Z\big)$ is given by the rank of vector bundles and the homomorphism $\widetilde{K}_0^{top}\big(V(\C)\big)\to H^2\big(V(\C),\Z\big)$ is the usual first Chern class of vector bundles.

\begin{lemma}\label{lemma:degen}
For any a smooth projective threefold $V$ over a separably closed field, the Atiyah--Hirzebruch spectral sequence degenerates in the $E_2$-term.
\end{lemma}
\begin{proof}
The only potentially non-zero differential in the Atiyah--Hirzebruch spectral sequence is the differential $d_3\colon H^3(V)\to H^6(V)$. However, since $V$ is a smooth projective threefold, the group $H^6(V)$ is torsion-free, whence this differential is zero.
\end{proof}

\medskip

We have a homomorphism of rings $K_0(V)\to KT_0(V)$, which commutes with pull-backs. Moreover, it commutes with push-forwards with respect to morphisms between smooth projective varieties. For topological $K$-groups, this was proved by Atiyah and Hirzebruch in~\cite[Theor.\,4.2]{AH}. For \'etale $K$-groups, this holds just by definition of the push-forward on them, which, in turn, is based on a comparison theorem by Thomason between algebraic \mbox{$K$-groups} and \'etale \mbox{$K$-groups}, see~\cite[\S\S\,1.13, 2.2]{Th2} (see also~\cite{Th3} for a more general definition of the push-forward on \'etale $K$-groups, which is also compatible with the push-forward on algebraic \mbox{$K$-groups}). We obtain a homomorphism between the rings of correspondences ${K_0(V\times V)\to KT_0(V\times V)}$, where the product is defined by composition of correspondences. Since the groups $KT_i(V)$, $i=0,1$, are modules over the ring of correspondences $KT_0(V\times V)$, we see that topological $K$-groups and \'etale $K$-groups are well-defined for~\mbox{$K$-motives}.

In what follows, let a ring $R$ be either~$\Z$ when $L=\C$, or $\Z_l$ when $L$ is an arbitrary separably closed field.

\begin{lemma}\label{lemma:KunittopkK}
Let $V$ be a smooth projective variety over a separably closed field $L$ such that the $K$-motive $KM(V)$ is of unit type. Then the natural map $K_0(V)_R\to KT_0(V)$ is an isomorphism and we have~$KT_1(V)=0$.
\end{lemma}
\begin{proof}
Let $r$ be such that $KM(V)\simeq \uno^{\oplus r}$. Since the groups $KT_0$ and $KT_1$ are well-defined for $K$-motives, we see that there are isomorphisms
$$
KT_i(V)\simeq KT_i(*)^{\oplus r}\,,\qquad i=0,1\,,
$$
where $*=\Spec(L)$ is the point. Further, we have
that $KT_0(*)=R$ and $KT_1(*)=0$ (for topological $K$-groups this is easily seen, while for \'etale $K$-groups this follows, for instance, from, the Atiyah--Hirzebruch spectral sequence).
\end{proof}

\medskip

\begin{lemma}\label{cor:torsioncohom}
Let $V$ be a smooth projective threefold over a separably closed field $L$ such that the $K$-motive $KM(V)$ is of unit type and let $l$ be a prime number different from the characteristic of~$L$. Then the following holds true:
\begin{itemize}
\item[(i)]
we have $H^i(V)=0$ for $i$ odd and $H^{i}(V)$ is torsion-free for $i$ even;
\item[(ii)]
the cycle class map ${CH^1(V)_R\to H^2(V)}$ is an isomorphism;
\item[(iii)]
we have $H^i_{\acute e t}(V,\Z/l)=0$ for $i$ odd and the canonical map
$H^{i}(V)/l\to H^i_{\acute e t}(V,\Z/l)$ is an isomorphism for $i$ even;
\item[(iv)]
the cycle class map $CH^1(V)/l\to H^2_{\acute e t}(V,\Z/l)$ is an isomorphism.
\end{itemize}
\end{lemma}
\begin{proof}
$(i)$ By Lemma~\ref{lemma:KunittopkK}, the group $KT_1(V)$ vanishes. Hence by Lemma~\ref{lemma:degen}, we see that $H^i(V)=0$ for $i$ odd. By Poincar\'e duality, the torsion subgroup of~$H^{i}(V)$ is Pontryagin dual to the torsion subgroup of~$H^{7-i}(V)$ for any $i\geqslant 0$. Thus $H^{i}(V)$ is torsion-free for $i$ even.

$(ii)$ By Proposition~\ref{prop:key}$(ii)$, the intersection pairing ${CH^1(V)_R\otimes CH^{2}(V)_R\to R}$ is non-degenerate. Hence the cycle class map ${CH^1(V)_R\to H^2(V)}$ is injective, because the intersection pairing factors through cohomology and $R\simeq H^6(V)$.

By Lemma~\ref{lemma:KunittopkK}, the natural map $\theta\colon K_0(V)_R\to KT_0(V)$ is an isomorphism. In addition, the map~$\theta$ respects (non-strictly) the filtration $F^{\bullet}K_0(V)_R$ by codimension of support and the filtration $F^{\bullet}KT_0(V)$ induced by the Atiyah--Hirzebruch spectral sequence. The corresponding map from ${\gr^0_FK_0(V)_R\simeq CH^0(V)_R\simeq R}$ to ${\gr^0_F KT_0(V)\simeq H^0(V)\simeq R}$ is the identity. Hence we have an isomorphism ${\theta\colon F^1K_0(V)_R\stackrel{\sim}\longrightarrow F^1KT_0(V)}$. This implies that the map from ${\gr^1_FK_0(V)_R\simeq CH^1(V)_R}$ to ${\gr^1_F KT_0(V)\simeq H^2(V)}$ is surjective. Also, one checks directly that this map is equal to the cycle class map.

$(iii)$ This is implied by $(i)$ and the universal coefficient theorem (recall that for complex varieties, \'etale cohomology groups with finite constant coefficients coincide with classical complex cohomology groups, see~\cite[Theor.\,III.3.12]{Milne}).

$(iv)$ This follows directly from~$(ii)$ and $(iii)$.
\end{proof}

\begin{remark}
If the characteristic of $L$ is zero, then Lemma~\ref{cor:torsioncohom} can be proved with the help of topological $K$-groups only. Indeed, the variety $V$ is defined over a field which is finitely generated over~$\bbQ$ and can be embedded into $\C$ and \'etale cohomology groups with torsion coefficients are invariant under extensions of algebraically closed fields of characteristic zero, see~\cite[Cor.\,VI.4.3]{Milne}.
\end{remark}

\medskip

Now we describe third \'etale cohomology of a threefold as in Theorem~\ref{thm:main}. All statements and arguments below make sense with coefficients $\Z/n$ after appropriate Tate twists, where $n$ is any natural number not divisible by the characteristic of the ground field, but \mbox{coefficients}~$\Z/2$ are enough for our purposes.

Given a field $F$, by $G_F$ we denote the absolute Galois group of~$F$. For short, by~$H^i(F,\Z/2)$ denote the canonically isomorphic groups $H^i(G_F,\Z/2)\simeq H^i_{\acute e t}\big(\Spec(F),\Z/2\big)$.

Let $V$ be a smooth variety over an arbitrary field $k$ of characteristic not~$2$. Consider the composition
\begin{equation}\label{eq:comp}
k^*/2\otimes CH^1(V)/2\longrightarrow H^1_{\acute e t}(V,\Z/2)\otimes H^2_{\acute e t}(V,\Z/2)\longrightarrow H^3_{\acute e t}(V,\Z/2)\,,
\end{equation}
where the first map is the tensor product of the Kummer isomorphism followed by the pull-back map
\begin{equation}\label{map1}
k^*/2\stackrel{\sim}\longrightarrow H^1(k,\Z/2)\longrightarrow H^1_{\acute e t}(V,\Z/2)
\end{equation}
and the cycle class map
\begin{equation}\label{map2}
CH^1(V)/2\longrightarrow H^2_{\acute e t}(V,\Z/2)
\end{equation}
and the second map is product in cohomology. Since $CH^1\big(\Spec(k(V))\big)=0$, the image of composition~\eqref{eq:comp} is contained in the kernel $NH^3_{\acute e t}(V,\Z/2)$ of the restriction map ${H^3_{\acute e t}(V,\Z/2)\longrightarrow H^3\big(k(V),\Z/2\big)}$. Thus we obtain a homomorphism
$$
\zeta\;:\; k^*/2\otimes CH^1(V)/2\longrightarrow NH^3_{\acute e t}(V,\Z/2)\,.
$$

\begin{prop}\label{prop:prod}
Let $X$ be an irreducible smooth projective variety of dimension $3$ over a field~$k$ such that the \mbox{$K$-motive}~$KM(X)$ is of unit type and the characteristic of $k$ is not $2$. Then we have an isomorphism
$$
\zeta\;:\; k^*/2\otimes CH^1(X)/2\stackrel{\sim}\longrightarrow NH^3_{\acute e t}(X,\Z/2)\,.
$$
\end{prop}
\begin{proof}
In order to compute $H^3_{\acute e t}(X,\Z/2)$ and $NH^3_{\acute e t}(X,\Z/2)$, we analyze the Hochschild--Serre spectral sequence
$$
{E_2^{ij}=H^i\big(G_k,H^j_{\acute e t}(X_{k^{sep}},\Z/2)\big)}\Rightarrow H^{i+j}_{\acute e t}(X,\Z/2)\,,
$$
where $k^{sep}$ is a separable closure of $k$.

Clearly, the $K$-motive of $X_{k^{sep}}$ is of unit type. Therefore, by Lemma~\ref{cor:torsioncohom}$(iii)$, we have the vanishing
$$
H^1_{\acute e t}(X_{k^{sep}},\Z/2)=H^3_{\acute e t}(X_{k^{sep}},\Z/2)=0\,.
$$
By Proposition~\ref{prop:key}$(iii)$ and Lemma~\ref{cor:torsioncohom}$(iv)$, the cycle class map defines an isomorphism
$$
CH^1(X)/2\stackrel{\sim}\longrightarrow H^2_{\acute e t}(X_{k^{sep}},\Z/2)\,.
$$
Since the action of $G_k$ on $CH^1(X)/2$ is trivial, we obtain the isomorphisms
$$
H^1\big(G_k,H^2_{\acute e t}(X_{k^{sep}},\Z/2)\big)\simeq H^1(k,\Z/2)\otimes CH^1(X)/2\simeq k^*/2\otimes CH^1(X)/2\,.
$$
Hence the $E_3$-term of the Hochschild--Serre spectral sequence looks as follows:
$$
\xymatrix{
0&0&0&0\\
CH^1(X)/2\ar[rrrdd]^{d_3}&k^*/2\otimes CH^1(X)/2&{*} &{*}\\
0&0&0&0\\
\Z/2&k^*/2&H^2(k,\Z/2)&H^3(k,\Z/2)}
$$
Note that the composition
$$
H^2_{\acute e t}(X,\Z/2)\longrightarrow H^0\big(G_k,H^2(X_{k^{sep}},\Z/2)\big)\stackrel{d_3}\longrightarrow H^3(G_k,\Z/2)
$$
vanishes. Since the isomorphism ${CH^1(X)/2\stackrel{\sim}\longrightarrow H^0\big(G_k,H^2(X_{k^{sep}},\Z/2)\big)}$ factors through the cycle class map $CH^1(X)/2\to H^2_{\acute e t}(X,\Z/2)$, we obtain that the differential ${d_3\colon CH^1(X)/2\to H^3(k,\Z/2)}$ vanishes. Thus we see that the Hochschild--Serre spectral sequence yields an exact sequence
$$
0\longrightarrow H^3(k,\Z/2)\longrightarrow H^3_{\acute e t}(X,\Z/2)\stackrel{\xi}\longrightarrow k^*/2\otimes CH^1(X)/2\,.
$$
It follows from multiplicativity of the Hochschild--Serre spectral sequence that the composition~$\xi\circ \zeta$ is equal (up to sign) to the identity and, in particular, $\xi$ is surjective.

We obtain a commutative diagram (up to sign) with exact column and raw
$$
\xymatrix{
&&0\ar[d]\\
&&NH^3_{\acute e t}(X,\Z/2)\ar[d]\\
0\ar[r]&H^3(k,\Z/2)\ar[r]\ar[rd]^{\gamma} &H^3_{\acute e t}(X,\Z/2)\ar[r]^{\xi}\ar[d]&k^*/2\otimes CH^1(X)/2\ar[r]\ar[lu]_{\zeta}&0\\
&&H^3\big(k(X),\Z/2\big)}
$$
Thus in order to prove that $\zeta$ is an isomorphism, it remains to show that the map $\gamma$ is injective. Recall that the unramified cohomology group is defined by the formula
$$
H^3_{nr}(X,\Z/2):=\bigcap_{D\subset X}\Ker\big( H^3(k(X),\Z/2)\to H^2(k(D),\Z/2)\big)\,,
$$
where the intersection is taken over all prime divisors $D\subset X$ and the maps ${H^3(k(X),\Z/2)\to H^2(k(D),\Z/2)}$ are residues, see, e.g,~\cite[\S\,4.1]{CT}. It follows from the localization exact sequence for \'etale cohomology that $\gamma$ factors as the composition
$$
H^3(k,\Z/2)\stackrel{\tilde\gamma}\longrightarrow H^3_{nr}\big(k(X),\Z/2\big)\hookrightarrow H^3\big(k(X),\Z/2\big)\,.
$$
By Proposition~\ref{prop:key}$(i)$, $X$ has a zero-cycle $\alpha$ of degree one. Since unramified cohomology groups are contravariant with respect to varieties, see~\cite[\S\,2.1]{CT}, the zero-cycle $\alpha$ defines the map
$$
\alpha^*\;:\; H^3_{nr}(X,\Z/2)\to H^3_{nr}\big(\Spec(k),\Z/2\big)=H^3(k,\Z/2)\,.
$$
The composition $\alpha^*\circ\tilde\gamma$ is the identity, whence $\tilde\gamma$ is injective and $\gamma$ is injective as well.
\end{proof}

\medskip

Now we pass to $K$-cohomology, that is, cohomology of sheaves of $K$-groups. Let $V$ be an irreducible smooth projective variety over $k$. Let $\cK_i$, $i\geqslant 0$, denote the Zariski sheaf on~$V$ associated with the presheaf that sends an open subset $U\subset V$ to the algebraic $K$-group $K_i(U)$ (in particular, $\cK_0=\underline{\Z}$ and $\cK_1=\cO_V^*$). By a result of Quillen in~\cite[\S\,7.5]{Qu}, $K$-cohomology groups~$H^i(V,\cK_{j})$ are canonically isomorphic to cohomology groups of a Gersten complex, which implies that ${H^i(V,\cK_{j})=0}$ if $i>j$ and that there are canonical isomorphisms
$$
H^i(V,\cK_i)\simeq CH^i(V)\,,\qquad i\geqslant 0\,.
$$
One has the Brown--Gersten spectral sequence
$$
E_2^{ij}=H^i(V,\cK_{-j})\Rightarrow K_{-i-j}(V)
$$
such that the arising filtration on $K$-groups is the filtration by codimension of support, see~\cite{BG}. Product between algebraic $K$-groups defines naturally product between $K$-cohomology groups and the Brown--Gersten spectral sequence is multiplicative in a natural sense (see, e.g.,~\cite[Theor.\,69]{Gil}).

For each $i\geqslant 0$, there are no differentials in the Brown--Gersten spectral sequence that come out of $H^i(V,\cK_i)$ and the arising map $CH^i(V)\simeq H^i(V,\cK_i)\to\gr^i_FK_0(V)$ coincides (up to sign) with the map~$\varphi_i$ from Proposition~\ref{prop:surjvarphi}. All differentials that come out of $H^0(V,\cK_1)\simeq k^*$ are zero by functoriality of the Brown--Gersten spectral sequence applied to the morphism $V\to\Spec(k)$. Thus the $E_2$-term looks as follows:
$$
\xymatrix{
\Z&0\\
k^*&CH^1(V)&0\\
H^0(V,\cK_2)\ar[rrd]^{d_2}&H^1(V,\cK_2)\ar[rrd]^{d_2}&CH^2(V)&0\\
H^0(V,\cK_3)&H^1(V,\cK_3)&H^2(V,\cK_3)&CH^3(V)&0}
$$

We see that there is an exact sequence
\begin{equation}\label{eq:Brown}
H^1(V,\cK_2)\stackrel{d_2}\longrightarrow CH^3(V)\stackrel{\varphi_3}\longrightarrow \gr^3_FK_0(V)\,.
\end{equation}
Since the kernel of the homomorphism $\varphi_3$ is $2$-torsion by Proposition~\ref{prop:surjvarphi}, the differential $d_2$ factors through the quotient~$H^1(V,\cK_2)/2$.

\medskip

By a result of Merkurjev and Suslin, see~\cite[\S\,18]{MS}, there is an exact sequence
\begin{equation}\label{eq:MS}
0\longrightarrow H^1(V,\cK_2)/2\stackrel{\nu}\longrightarrow NH^3_{\acute e t}(V,\Z/2)\longrightarrow CH^2(V)_2\longrightarrow 0\,,
\end{equation}
where, as above, $CH^2(V)_2$ denotes the $2$-torsion subgroup of $CH^2(V)$.

Let us describe the map $\nu$ in more detail. Let $\cH^i$ denote the Zariski sheaf on $V$ associated with the presheaf that sends an open subset $U\subset V$ to \'etale cohomology $H^i_{\acute e t}(U,\Z/2)$. We have the \'etale Chern classes $\cK_i\to \cH^i$, see~\cite[\S\,II.2.3]{Sou}, which define the corresponding map between cohomology
\begin{equation}\label{MS1}
H^1(V,\cK_2)/2\longrightarrow H^1(V,\cH^2)\,.
\end{equation}
Note that for $i=1$ the \'etale Chern class is given by the Kummer theory, while for $i=2$ this is the norm residue symbol on decomposable elements.
Further, the direct image of sheaves from the \'etale topology to the Zariski topology defines the Leray spectral sequence
$$
E_2^{ij}=H^i(V,\cH^{j})\Rightarrow H^{i+j}_{\acute e t}(V,\Z/2)\,.
$$
By the main result of Bloch and Ogus in~\cite{BO}, we have $H^i(V,\cH^j)=0$ for $i>j$. Therefore the $E_2$-term of the Leray spectral sequence looks as follows:
$$
\xymatrix{
H^0(V,\cH^3)\ar[rrd]^{d_2}&H^1(V,\cH^3)&H^2(V,\cH^3)\\
H^0(V,\cH^2)&H^1(V,\cH^2)&H^2(V,\cH^2)\\
H^0(V,\cH^1)&H^1(V,\cH^1)&0\\
\Z/2&0&0}
$$
Thus we obtain an injective map
\begin{equation}\label{MS2}
H^1(V,\cH^2)\hookrightarrow H^3_{\acute e t}(V,\Z/2)\,.
\end{equation}
The composition of the maps~\eqref{MS1} and~\eqref{MS2} gives $\nu$. Also, note that the compositions
$$
k^*/2\stackrel{\sim}\longrightarrow H^0(V,\cK_1)/2\longrightarrow H^0(V,\cH^1)\stackrel{\sim}\longrightarrow H^1_{\acute e t}(V,\Z/2)\,,
$$
$$
CH^1(V)/2\stackrel{\sim}\longrightarrow H^1(V,\cK_1)/2\longrightarrow H^1(V,\cH^1)\hookrightarrow H^2_{\acute e t}(V,\Z/2)
$$
coincide with the maps~\eqref{map1} and~\eqref{map2}, respectively.

\medskip

Proposition~\ref{prop:prod} implies the following result.

\begin{cor}\label{cor:notors}
Under assumptions of Proposition~\ref{prop:prod}, the group $CH^3(X)$ is torsion-free.
\end{cor}
\begin{proof}
The composition of the product map between $K$-cohomology
$$
\mu\;:\;k^*/2\otimes CH^1(X)/2\longrightarrow H^1(X,\cK_2)/2
$$
with the injective map~${\nu\colon H^1(X,\cK_2)/2\to NH^3_{\acute e t}(X,\Z/2)}$ from the exact sequence~\eqref{eq:MS} is equal (up to sign) to the isomorphism~$\zeta$ from Proposition~\ref{prop:prod} (this follows from multiplicativity of the \'etale Chern classes and the multiplicativity of the Leray spectral sequence). Hence~$\mu$ is an isomorphism (also, we obtain that~$\nu$ is an isomorphism as well, which follows alternatively from Proposition~\ref{prop:key} and the exact sequence~\eqref{eq:MS}).

Recall that the differential $d_2\colon H^1(X,\cK_2)\to CH^3(X)$ factors as a composition
$$
H^1(X,\cK_2)\longrightarrow H^1(X,\cK_2)/2\longrightarrow CH^3(X)\,.
$$
Since $\mu$ is an isomorphism and the Brown--Gersten spectral sequence is multiplicative, we see that the differential $d_2\colon H^1(X,\cK_2)\to CH^3(X)$ vanishes. Thus the exact sequence~\eqref{eq:Brown} implies that $\varphi_3$ is injective and $CH^3(X)$ is torsion-free, because so is $K_0(X)$.
\end{proof}

\medskip

Now we are ready to prove our main result.

\begin{proof}[Proof of Theorem~\ref{thm:main}]
We may assume that $X$ is irreducible because, clearly, a direct summand of a unit type \mbox{$K$-motive} is of unit type as well (look at the endomorphisms algebra of a unit type \mbox{$K$-motive}).

By Remark~\ref{rmk:lawdim}, we need to consider the case when $X$ is a threefold. By Corollary~\ref{cor:split}, we have an isomorphism
$M(X)\simeq M\oplus N$, where the Chow motive $M$ is of Lefschetz type and for any field extension $k\subset L$, the only non-trivial Chow group of the Chow motive $N_L$ is $CH^3(N_L)$, which coincides with $CH^3(X_L)_2$. By Corollary~\ref{cor:notors}, this group vanishes as well, because the $K$-motive of $X_L$ in $\KM(L)$ is of unit type. Now by Remark~\ref{rmk:Chowvan}, we have $N=0$, which finishes the proof of the theorem.
\end{proof}

\end{document}